\documentclass{amsart}

\usepackage{latexsym}
\usepackage{verbatim}
\usepackage{amsmath}
\usepackage{amssymb}
\usepackage{mathrsfs}
\usepackage{amsthm}
\usepackage{tikz}
\usetikzlibrary{matrix}

\usetikzlibrary{decorations.pathreplacing}
\usetikzlibrary{arrows.meta}

\usepackage{amsthm, amsfonts}
\usepackage{euscript}
\usepackage{graphicx}
\usepackage{float}
\usepackage{caption}
\usepackage{tikz-qtree}
\usetikzlibrary{arrows.meta}
\usepackage{subcaption}

\newcommand{\g}{\gamma}

\newtheorem{defi}{Definition}[section]
\newtheorem{theorem}[defi]{Theorem}

\newtheorem{lemma}[defi]{Lemma}
\newtheorem{proposition}[defi]{Proposition}
\newtheorem{conjecture}[defi]{Conjecture}
\newtheorem{corollary}[defi]{Corollary}
\newtheorem{problem}[defi]{Problem}
\newtheorem{remark}[defi]{Remark}

\newtheorem{examples}[defi]{Examples}
\newtheorem{question}[defi]{Question}

\def\Aut{\mathop{\rm Aut}\nolimits}

\def\nil2{\mathop{\rm nil2}\nolimits}
\def\Th{\mathop{\rm Th}\nolimits}
\def\exp{\mathop{\rm exp}\nolimits}
\def\sup{\mathop{\rm sup}\nolimits}
\def\h{\mathop{\rm ht}\nolimits}
\def\Hom{\mathop{\rm Hom}\nolimits}

\begin{document}
\author{Dugald Macpherson \and Katrin Tent}
\title{Omega-categorical pseudofinite groups}
\begin{abstract}
We explore the interplay between  $\omega$-categoricity and pseudofiniteness for groups, conjecturing that $\omega$-categorical pseudofinite groups are finite-by-abelian-by-finite. We  show that the conjecture reduces to nilpotent $p$-groups of class~2, and give a proof that several of  the known examples of $\omega$-categorical $p$-groups satisfy the conjecture. In particular, we show by a direct counting argument that for any odd prime $p$ the ($\omega$-categorical) model companion of 
the theory of nilpotent class 2 exponent $p$ groups, constructed by Saracino and Wood, is not pseudofinite, and that an $\omega$-categorical  group constructed by Baudisch with supersimple rank 1 theory  is not pseudofinite. We also survey some scattered literature on $\omega$-categorical groups over 50 years. 
\end{abstract}

\maketitle
\section{Introduction}
This paper explores the interplay between two  model-theoretic finitary conditions on infinite groups: $\omega$-categoricity and pseudofiniteness. We recall that a countably infinite first order structure is {\em $\omega$-categorical} if any countably infinite $N$ with the same first order theory as $M$ is isomorphic to $M$, or equivalently, if $\Aut(M)$ acts {\em oligomorphically} on $M$, that is, has finitely many orbits on $M^n$ for all $n$. We say a structure $M$ is {\em pseudofinite} if it is infinite and every first order sentence true of $M$ has a finite model. The paper is motivated by the following conjecture.

\begin{conjecture} \label{mainq} Every $\omega$-categorical pseudofinite group $G$ is definably finite-by-abelian-by-finite, that is, has a definable normal subgroup $N$ of finite index which has finite derived subgroup.
\end{conjecture}

We shall describe such $G$ as `finite-by-abelian-by-finite', and note that this is unambiguous: any  finite-by-(abelian-by-finite) group is clearly (finite-by-abelian)-by-finite, and conversely, if $G$ has a  finite-by-abelian normal subgroup $N$ of finite index, then its derived subgroup $N'$ is finite and characteristic in $N$ so normal in $G$, and $G$ is   finite-by-(abelian-by-finite). Furthermore, as noted in the introduction to \cite{ev-wagner}, we may take $N$ to be characteristic. 
 The upshot is that we may suppose that the normal subgroups of $G$ witnessing `finite-by-abelian-by-finite'  are characteristic and hence (by $\omega$-categoricity) definable without parameters, and from now on we just call such a group $G$ an {\em FAF group}.

We show that Conjecture~\ref{mainq} reduces to the case of $\omega$-categorical pseudofinite   $p$-groups of nilpotency class at most~$2$:

\begin{proposition} \label{reduction}
Conjecture~\ref{mainq} holds if
 every $\omega$-categorical pseudofinite  $p$-group of class at most 2 is FAF. 
\end{proposition}

It is easy to give examples of $\omega$-categorical pseudofinite  groups which are abelian (e.g. elementary abelian $p$-groups), and which are  finite-by-abelian  but not abelian-by-finite (see also the examples in~\ref{examples}). For the latter, the extraspecial $p$-groups of exponent $p$ ($p$ odd) provide examples: they are $\omega$-categorical by \cite{felgner}, and pseudofinite since they are smoothly approximated (in the sense of \cite{klm}) by finite extraspecial $p$-groups, as noted in \cite[`Added in proof']{klm}. See \cite[Appendix A]{milliet} for further discussion of the model theory of  extraspecial groups. The construction of extraspecial groups is generalised in \cite[Theorem A]{apps1}, where it is shown that if $G$ is a finite nilpotent class~2 group and $K$ is a subgroup with $G'\leq K\leq Z(G)$, then the central product of $\aleph_0$ copies of $G$, amalgamated over $K$, is $\omega$-categorical. 

We give a short proof in Section 2 of the following result, an easy consequence of known results on $\omega$-categorical groups.

\begin{proposition} \label{nilpot}
Every $\omega$-categorical pseudofinite group is nilpotent-by-finite.
\end{proposition}

We then show that Conjecture~\ref{mainq} holds for many  of the $\omega$-categorical class~2 groups that we know of, in particular (for odd primes $p$) for the model companion of the theory of exponent $p$ class~2 $p$-groups described by Saracino and Wood in  \cite{sar-wood}, and various generalisations; namely  the {\em comprehensive} groups of Apps \cite{apps1} (see Theorems~\ref{comp} and ~\ref{comp_not_psf}) and higher nilpotency class generalisations of the Saracino-Wood groups, considered in \cite{maier},  \cite{baudisch2}  \cite{baudisch-neo}, and \cite{DMRS}. In particular, we prove the following result (Corollary~\ref{nil-pseud} below).

\begin{theorem} \label{sarwood}
\begin{enumerate}
\item[(i)]
For $p$ an odd prime, the unique countable existentially closed nilpotent class~2 group of exponent $p$ is not pseudofinite.
\item[(ii)] The $\omega$-categorical groups $D(n)$ with supersimple rank 1 theory considered by Baudisch in \cite{baudisch-neo} are not pseudofinite for $n\geq 2$.
\end{enumerate}
\end{theorem}

We remark that Baudisch \cite{baudisch-neo} showed that the group in (i) has TP${}_2$ theory. More recently, d'Elb\'ee, M\"uller, Ramsey and Siniora \cite[Corollary 3.7]{DMRS} have shown it to have an NSOP${}_1$ theory. Part (ii) above shows that not every $\omega$-categorical finite-by-abelian group is pseudofinite, i.e. the converse to our conjecture is false: the groups $D(n)$ for $n$ finite are all finite-by-abelian.

A theory has the {\em strict order property} if it has a model $M$ such that there is a definable preorder on some power $M^n$ with an infinite totally ordered subset, and a model has the strict order property if its theory has. We say the model (or theory) is {\em NSOP} otherwise. 
It is  easy to see that an $\omega$-categorical  structure with the strict order property is not pseudofinite (see Lemma~\ref{SOP}).  The obvious ways to show that a structure is not pseudofinite are to define a partial ordering which is dense or has no greatest element, or to define a map $f:X \to X$ which is injective but not surjective, or vice-versa. Such a function cannot be definable in an $\omega$-categorical structure, since it would lead to pairs from $X$ at arbitrary `distance' with respect to $f$. We believe Theorem~\ref{sarwood} may be the first result proving non-pseudofiniteness of an $\omega$-categorical structure without the strict order property, and that (ii) may be the first example of an $\omega$-categorical supersimple non-pseudofinite structure. (In the other direction, Kruckman in \cite[Theorem 4.5]{kruckman} shows that the theory $T_{{\rm feq}}^{*}$ of parametrised equivalence relations is $\omega$-categorical  pseudofinite but not supersimple.)
Of course, pseudofiniteness is unknown for many $\omega$-categorical structures -- for example it is a well-known open problem posed by Cherlin for the universal homogeneous triangle-free graph (see \cite[Problem A']{cherlin1}, or \cite{cherlin2}). Pseudofiniteness seems also to be open for the arity 3 analogue of the  homogeneous triangle-free graph -- namely the universal homogeneous tetrahedron-free 3-hypergraph, which  (unlike the non-simple homogeneous triangle-free graph) has supersimple rank 1 theory. 

\medskip

{\bf Some background on $\omega$-categorical groups.}
  The structure of $\omega$-categorical groups has been revisited since the early 1970s by many authors, and some substantial literature appears to be little known. We summarise some aspects here, focussing on themes related to Conjecture~\ref{mainq}.

First,  there are several results  recovering our intended conclusion, {\em  i.e.}  showing that $\omega$-categorical groups satisfying certain model-theoretic hypotheses are FAF.  For example, by \cite{bcm} every $\omega$-categorical superstable group is abelian-by-finite (but under stability only virtual nilpotency is known). This is generalised in \cite{ev-wagner}, where it is shown that every $\omega$-categorical  group with supersimple theory  is FAF. The latter also generalises Proposition 6.2.4 from \cite{cher-hrush} where the same conclusion is obtained for groups interpretable in a smoothly approximable structure; such structures are known to be supersimple. (In \cite[Proposition 6.2.4]{cher-hrush}, there are assumptions of type amalgamation and modularity -- these are shown to hold for smoothly approximable structures in Proposition 5.1.15 and Corollary 5.6.4 respectively -- note that smooth approximation is equivalent to Lie coordinatisability by \cite[Theorem 2]{cher-hrush}.) Finally, by \cite{dob-wag}, every $\omega$-categorical group of finite {\em burden} is FAF. In each case, there is some notion of rank, and the arguments often proceed by identifying minimal abelian groups and working with definable isogenies. Some of these results have ring-theoretic analogues, and as noted at the end of this paper, Conjecture~\ref{mainq} has a ring-theoretic equivalent, namely  Conjecture~\ref{qrings}.

An important construction technique for $\omega$-categorical algebraic objects is that of {\em boolean powers}. For groups, one takes an $\omega$-categorical or finite group $G$ equipped with the discrete topology, the Stone space $S$ of a countable $\omega$-categorical Boolean algebra $B$ (for example the countable atomless Boolean algebra, in which case $S$ is a Cantor space), and forms the group $B[G]$ consisting of all continuous maps $S\to G$, with coordinatewise multiplication as the group operation. There is an analogous construction (see Section 2 of \cite{apps2}, or \cite[Section 2]{wilson}) when $B$ is an $\omega$-categorical Boolean ring without a 1: in the case above corresponding to the Cantor set, one obtains the group $B^-(G)$ consisting of all $f:S \to G$ with $f(x_0)=1_G$, where $x_0\in S$ is specified.  It is easily seen that if $G$ is an $\omega$-categorical group which is nilpotent of class~2 but not abelian, and $B$ is infinite, then $B[G]$ and $B^-[G]$ are $\omega$-categorical nilpotent class~2 and not FAF (see  for example  the paragraph before Proposition 4.1 in \cite{apps1}).  By the  easy Lemma~\ref{boolpseud} below, such groups are not pseudofinite. For more on Boolean powers of groups, see also \cite{apps0}. In particular, since any $\omega$-categorical group has a maximally refined chain of characteristic subgroups with characteristically simple and $\omega$-categorical or finite successive quotients, \cite[Theorem A]{apps2} gives an approach to structural questions about $\omega$-categorical groups.

A number of authors have revisited the subject of $\omega$-categorical nilpotent groups, from the viewpoints of Fra\"iss\'e limits, existentially closed groups in certain classes,  and model companions. We summarise briefly what is known.

Saracino and Wood \cite[Theorem 3.9]{sar-wood}, among other results, show that for every finite $m\geq 2$ the collection of nilpotent class~2 groups of exponent $m$ has a model companion, and this model companion is $\omega$-categorical. The proof is essentially by identifying appropriate axioms (of the form generalised  by Apps, see below). A related thread of work explores groups whose theory has quantifier-elimination. Work of Cherlin and Felgner \cite{cher-felg} reduces the classification to that of QE nilpotent class~2 groups of exponent 4. In \cite{sar-wood2}, constructions are given of $2^{\aleph_0}$ non-isomorphic countable $\omega$-categorical such groups. The construction is analogous to that of Henson \cite{henson} of homogeneous digraphs: the authors find a family $\mathcal{F}$ of mutually non-embeddable finite  class~2 groups of exponent  4, for each $S\subset \mathcal{F}$ consider the class of finite  class~2 groups of exponent at most 4  not embedding a member of $S$, and show this class has the amalgamation property;  distinct sets $S$ yield non-isomorphic Fra\"iss\'e limits. This approach is developed further in \cite{cher-sar-wood} with a view also to constructions of homogeneous rings. 

Apps \cite{apps1}, generalising the Saracino-Wood construction of model companions for the theory of class~2 $p$-groups with a given bound on the exponent (not necessarily prime), gives the construction of {\em comprehensive groups} described in Section 2 below. The key point is that the centre $Z(G)$ is specified in advance, together with the exponents of $Z(G)$, $G/Z(G)$, and $G$. For given data satisfying certain conditions, the model companion of such a class is described in \cite[Theorem D]{apps1}. 

Maier in the late 1980s wrote a series of papers on constructions of existentially closed nilpotent groups in various classes. The proofs are based on an intricate adaptation, using Lazard series,   of Higman's amalgamation results for nilpotent groups, where the groups are equipped with a common central series. In particular, in \cite[Theorem 3.5]{maier}, Maier shows that for $p$ an odd prime and  $c<p$ the class of groups of nilpotency class $c$ and exponent $p$ has a unique existentially closed member, its model companion, which is $\omega$-categorical. The main focus of Maier's paper is elsewhere, and he notes that in the case $c=2$, Saracino and Wood had already constructed this model companion in \cite[Theorem 3.9]{sar-wood}.

Baudisch \cite{baudisch2} shows that for an odd prime $p$ the class of finite groups of exponent $p$ and class at most 2, equipped with a predicate for a central subgroup, has the amalgamation property. Its Fra\"iss\'e limit will as a group be isomorphic to the model companion of the collection of class~2 exponent $p$ groups. It follows that the latter group has quantifier-elimination once a predicate is added for its centre, and Baudisch also shows that this group does not have simple theory, since there is an infinite descending chain of centralisers of finite sets, each with infinite index in its predecessor. In more recent work Baudisch also shows in \cite{baudisch-neo} that its theory is TP${}_2$. He explores in \cite{baudisch3, baudisch4}  generalisations of the construction in \cite{baudisch2} to arbitrary nilpotency class $c<p$.
  The approach to amalgamation is analogous to that of Maier, and there is a precursor (for $c=2$) in \cite{saracino}.

Very recently, d'Elb\'ee, M\"uller, Ramsey and Siniora \cite{DMRS} have revisited the results of Baudisch and Maier. They give a careful description of the amalgamation in the context of Lie algebras $L$ equipped with a `Lazard series', that is, a series $L=L_0\geq L_1 \geq \ldots \geq L_{c+1}=0$, with $[L_i,L_j]\leq L_{i+j}$ for all $i,j$ (where $L_k=0$ for all $k>c$). They
 transfer the amalgamation from such `Lazard Lie algebras' to  the corresponding class of finite groups (exponent $p$, class at most $c$, where $c<p$, equipped with a Lazard series). This yields a Fra\"iss\'e limit $\mathbf{G}_{c,p}$, which is $\omega$-categorical, and exactly the model companion of the class of exponent $p$ groups of class at most $c$ constructed by Maier in \cite[Theorem 3.5]{maier}. 
They show that the groups $\mathbf{G}_{2,p}$ are NSOP${}_1$ -- it was already known from Baudisch's work that they do not have supersimple theory. Strikingly, they also show that for $c\geq 3$ the group $G_{c,p}$ is SOP${}_3$ and NSOP${}_4$, and is $c$-dependent and $(c-1)$-independent.

\medskip

{\bf Structure of the paper.} Propositions~\ref{reduction} and ~\ref{nilpot}, along with Theorem~\ref{sarwood}, are proved in Section 2. We also show in Theorem~\ref{comp_not_psf} that Apps's comprehensive groups are not pseudofinite provided they have non-cyclic centre. The section includes a discussion of pseudofiniteness for $\omega$-categorical FAF groups. In Section~3 we note that the Boolean power construction does not yield counterexamples to our conjecture, that a counterexample will not arise from an ultraproduct of finite $p$-groups with large automorphism groups, and present the ring-theoretic analogue of the conjecture.

\medskip

{\bf Acknowledgement:} We warmly thank Bettina Eick for very helpful discussions related to the proof of Corollary~\ref{cor_nilp}, and Laurent Barthodi, for allowing us to include the proof of Lemma~\ref{laurent} which he provided. We also thank Christian d'Elb\'ee for several very helpful discussions.

\section{Saracino-Wood groups and variants}

In this section we reduce our conjecture to the case of nilpotent class~2 groups, and then show that several variants of the Saracino-Wood existentially closed nilpotent class~2 exponent $p$ group (for $p$ an odd prime) are not pseudofinite. We do not consider the case $p=2$ (where there are analogues of exponent 4). 
Since it is short and central to the paper, we first  include a proof of the following easy result. See also \cite[Proposition 1.3]{kruckman}, where the result is described as folklore.
\begin{lemma}\label{SOP} Let $M$ be an $\omega$-categorical pseudofinite structure. Then $M$ does not have the strict order property.
\end{lemma}
\begin{proof} 
Suppose that $M$ is $\omega$-categorical with the strict order property. Then (possibly naming parameters by constants) there is a 0-definable preorder $<$ on some power $M^d$ of $M$, with an infinite chain. By $\omega$-categoricity there are finitely many 2-types, so by Ramsey's theorem  there is a formula $\phi(x,y)$ isolating a complete 2-type and implying the formula $x<y$, and $\{a_i:i\in \omega\}\in M^d$ such that $\phi(a_i,a_j)$ holds whenever $i<j$. In particular, some sentence $\sigma$ of $\Th(M)$ expresses that $<$ is a  preorder, that $\phi(x,y)$ implies $x<y$, that $\exists x\exists y\phi(x,y)$, and that  $\forall x\forall y\big(\phi(x,y) \to \exists z(\phi(x,z)\wedge \phi(z,y))\big)$;  the last implication holds since $M\models \phi(a_1,a_3) \wedge \exists z(\phi(a_1,z)\wedge \phi(z,a_3))$, and because $\phi$ isolates a complete type.  Such a sentence clearly has no finite model. 
\end{proof}

We thus obtain a first step (Proposition~\ref{nilpot} of the introduction) towards proving Conjecture~\ref{mainq}. 
\begin{corollary}\label{SOPgroup}
Let $G$ be an $\omega$-categorical pseudofinite group. Then $G$ has a nilpotent 0-definable normal subgroup of finite index.
\end{corollary} \label{nilp}
\begin{proof} By Lemma~\ref{SOP}, $G$ does not have the strict order property. The result now follows immediately from \cite[Theorem 1.2]{mac1}. 
\end{proof}

For any nilpotent group $G$ of class $c$ we write $$G=\g_1G>\g_2G>\cdots >\g_{c+1}G=1$$
for the lower central series of $G$. 

\begin{remark} \rm \label{commutators}
In arguments below we shall use that if $H$ is a nilpotent group of class 2 and $x,y\in H$ then for any $n\geq 1$ we have 
$[x,y]^n=[x^n,y]$. This is proved easily by induction. In particular, the exponent of $\g_2H$ divides that of $\g_1H/\g_2H$. 
We also use the identities $[xz,y]=[x,y]^z[z,y]$ and $[x,y]^{-1}=[y,x]$, the latter giving $[H,G]=[G,H]$ whenever $H$ is a subgroup of $G$.
\end{remark}

We aim next for Corollary~\ref{cor_nilp}, which is just Proposition~\ref{reduction} from the Introduction. First, we give a lemma, for which we are very grateful to Laurent Barthodi for supplying the proof.

\begin{lemma} \label{laurent}
Let $G$ be a $p$-group of class 3, and let $|\g_2G/\g_3G|=p^m$. Then $|\g_3G|\leq p^{2m^3}$. In particular $G$ is finite-by-abelian.
\end{lemma}

\begin{proof} 
 First observe that $\g_2G/\g_3G$ has exponent at most $p^m$, and every generating set has size at most $m$. Thus $\g_2G/\g_3G$  is generated (modulo $\gamma_3G$) by at most $m$ commutators $[a,b]$ in $G$. Let $H\leq G$ be a subgroup generated modulo $\g_3G$ by a set $S$ of size  at most $2m$ of these $a, b$.  Then $\g_2G\leq H$, so $H$ is normal in $G$. Thus $$[H,H]\g_3G=\g_2G=[H,G]\g_3G=[G,G].$$
For $x,y\in H$ and $z\in G$, the Hall-Witt identity
$$ {\displaystyle [[x,y^{-1}],z]^{y}\cdot [[y,z^{-1}],x]^{z}\cdot [[z,x^{-1}],y]^{x}=1}$$
 and the fact that $H$ is normal in $G$, it now follows that
 $$[[H,H],G] \leq [[H,G],H]\cdot [[G,H],H]=[[G,H],H].$$

Since $\g_3G\leq Z(G)$ we now have
\begin{align*}
\gamma_3G  &= [\g_2G,G] = [[H,H]\gamma_3G,G]=[[H,H],G]\leq [[G,H],H]\\
& =[[G,H]\g_3G,H]=[[H,H],H].
\end{align*}

Note that since $\g_3G\leq Z(G)$, for $x,z\in S$ and $y\in \g_2G$ we have $[x,y]\in Z(G)$ and so $[xz,y]=[x,y]^z[z,y]=[x,y][z,y]$. Hence  $\gamma_3G$ is generated by at most $2m^2$ commutators of the form $[x,y]$ with $x\in S$ a generator of $H$ and $y\in \g_2G$ a generator of $\gamma_2G/\gamma_3G$.  By Remark~\ref{commutators} (applied to the class 2 group  $\g_2G$), $\g_3H$ has exponent at most $p^m$. Thus  $|\g_3G|\leq p^{f(m)}$ where $f(m)=2m^3$.
\end{proof}

\begin{corollary}\label{cor_nilp}
Suppose (A) every $\omega$-categorical pseudofinite  $p$-group of class at most 2 is FAF. 
\\Then the same conclusion holds for all $\omega$-categorical pseudofinite groups.
\end{corollary}

\begin{proof} In the argument below, we always assume that normal subgroups witnessing the FAF condition are characteristic and hence 0-definable; see the discussion of FAF groups after Conjecture~\ref{mainq}.

 Let $G$ be an $\omega$-categorical pseudofinite group. By Corollary~2.\ref{nilp} and the last paragraph we may suppose that $G$ is nilpotent. As any periodic nilpotent group is a direct product of its Sylow subgroups, we may further suppose that $G$ is a $p$-group for some prime $p$.
 
Observe that under (A), any pseudofinite $\omega$-categorical $p$-group with a class~2 definable normal subgroup of finite index, or with a finite normal subgoup with class~2 quotient, is FAF. For by (A), the class~2 piece is FAF, and absorbs the finite piece at the top or bottom.

Since there are only finitely many characteristic subgroups, we may suppose by the previous paragraph that $G$ has no proper characteristic subgroup  of finite index. Likewise, we may suppose that $G$ has no non-trivial finite characteristic subgroup. Hence we aim to show that $G$ is abelian.
 
  We now use induction on the nilpotency class of $G$ and let $1\neq \g_kG\leq Z(G)$ be the smallest non-trivial term of the lower central series. We need to show $k=1$. By assumption (A) and our reduction of $G$ we know that $k\neq 2$. So suppose towards a contradiction that $k\geq 3$. By induction and since $G$ has no proper characteristic subgroup of finite index, $\bar{G}=G/\g_kG$ has finite commutator group $\bar{G}'$ and $\bar{G}= C_{\bar{G}}(\bar{G}')$. Thus $G$ has nilpotency class 3.
Consider  the lower central series
  $G=\g_1G>\g_2G>\g_3G>1$ and note that $\g_2G/\g_3G\cong\bar{G}'$ is finite.
  By Lemma~\ref{laurent}, $G$ is finite-by-abelian, with a non-trivial finite characteristic subgroup, contradicting our assumption.
  Hence $k=1$ and $G$ is abelian.
\end{proof}

Our next goal is to show that certain constructions of $\omega$-categorical $p$-groups discussed  in the papers  \cite{apps1,baudisch1,baudisch2,baudisch3,sar-wood, cher-sar-wood,maier} are not pseudofinite. In order to make the idea as explicit as possible in the simplest case, we first consider existentially closed class~2 $p$-groups of exponent $p$. 

Let $p$ be an odd prime. By \cite[Theorem 3.9]{sar-wood}, the theory $T_{\nil2,p}$ of existentially closed nilpotent class~2 exponent $p$ groups is complete, and is axiomatised by a sentence $\rho$ expressing that a model $G$ is a nilpotent, class~2 exponent $p$ group and that the centre is the set of commutators, an axiom scheme expressing that the centre is infinite, and the following axiom scheme $\Sigma_n:=\{\sigma_n:n>0\}$. Here $\sigma_n$ expresses: 
$$(\forall \mbox{~linearly independent~} g_1,\ldots,g_n\in G\setminus Z(G))( \forall h_1,\ldots,h_n\in Z(G))(\exists k\in G)$$ $$(\bigwedge_{i=1}^n [g_i,k]=h_i).$$
Here linear independence is interpreted in the sense of the $\mathbb{F}_p$-vector space $G/Z(G)$.
As noted in \cite[Theorem 3.9]{sar-wood}, $T_{\nil2,p}$ is $\omega$-categorical. However, the following lemma yields that $\rho\wedge\sigma_1\wedge\sigma_2\wedge \tau$ has no finite model, where $\tau$ expresses that the centre has order at least $p^2$. 

\begin{lemma}\label{bil}
Let $p$ be an odd prime, and let $(V,W,\beta)$ be a 2-sorted structure where $V,W$ are  vector spaces over $\mathbb{F}_p$ (each sort viewed in the usual language of $\mathbb{F}_p$-modules) of dimension at least 2 and $\beta:V\times V \to W$ is bilinear. Let $\chi$ be a sentence expressing this information, and $\psi$ be the sentence
$$(\forall \mbox{~linearly independent~}v_1,v_2\in V) (\forall w_1,w_2\in W)$$ $$ (\exists z\in  V)(\beta(v_1,z)=w_1\wedge \beta(v_2,z)=w_2).$$
Then $\chi\wedge \psi$ has no finite model.
\end{lemma}
\begin{proof} Suppose for a contradiction that $(V',W',\beta)$ is a finite model of $\chi\wedge \psi$, with $\dim(V')=n>1$ and $\dim(W')=d>1$. For each $v\in V'$, the map $\beta_v:V'\to W'$ given by $\beta_v(x)=\beta(v,x)$ is linear, and is surjective as $(V',W',\beta)\models \psi$. Thus the kernel $K_v$ of $\beta_v$ has index $p^d$. Furthermore, $\psi$ ensures that if $v_1,v_2\in V'$ are linearly independent then any two cosets $K_{v_1}+a_1$ and $K_{v_2}+a_2$ intersect: indeed, if $w_i=\beta(v_i,a_i)$ for $i=1,2$ then  because  $(V',W',\beta)\models\psi$ there is $z\in V'$ with $\beta(v_1,z)=w_1\wedge \beta(v_2,z)=w_2$, and we have $z\in (K_{v_1}+a_1) \cap (K_{v_2}+a_2)$.
Thus $K_{v_1}+K_{v_2}=V'$, for if $z\in V'$ there is $u_1\in K_{v_1}\cap (K_{v_2}+z)$, so $u_1=u_2+z$ for some $u_2\in K_{v_2}$, and $z=u_1-u_2\in K_{v_1}+K_{v_2}$. Thus, we have a collection consisting of $ \frac{p^n-1}{p-1}$ such kernels $K_v$, each of codimension $d$ in $V'$, such that any two span $V'$. By moving to the dual space $(V')^*$ and taking annihilators, we find a family  consisting of $\frac{p^n-1}{p-1}$ subspaces of $(V')^*$ each of dimension $d$, with pairwise intersection $\{0\}$. Since $d>1$, this is clearly impossible by counting vectors in $(V')^{*}$.

\end{proof}

\begin{corollary}\label{bil-to-group}
Suppose that $G$ is an $\omega$-categorical nilpotent group of class~2  and exponent $p$ (where $p$ is an odd prime)  with both the centre $Z$ and the group $G/Z$  elementary abelian of rank at least 2, and such that $G\models \sigma_1\wedge \sigma_2$.  Then $G$ is not pseudofinite.

\end{corollary}

\begin{proof}This follows immediately from Lemma~\ref{bil}. Indeed, let $G\models T_{\nil2,p}$, and  put $V=G/Z(G)$ and $W=Z(G)$ -- these are elementary abelian $p$-groups and so vector spaces over $\mathbb{F}_p$. Let $\beta$ be the commutator map $G/Z(G)\times G/Z(G) \to W$ given by $[gZ(G),hZ(G)]=[g,h]$. The axioms $\rho \wedge \sigma_1\wedge\sigma_2$ yield that $(V,W,\beta)\models  \chi \wedge \psi$.
\end{proof}

This yields almost immediately the following result, and hence gives Theorem~\ref{sarwood}. The groups $D(n)$ (for $n$ finite) are $\omega$-categorical nilpotent class~2 and exponent $p$, with centre elementary abelian of rank $n$. By \cite[Theorem  3]{baudisch-neo}, they are supersimple of SU-rank 1. Note that for $n=1$ the group is extraspecial, so pseudofinite.

\begin{corollary}\label{nil-pseud}
Let $p$ be an odd prime.
\begin{enumerate}
\item[(i)] The theory $T_{\nil2,p}$ is not pseudofinite.
\item[(ii)] For $n\geq 2$ the group $D(n)$ discussed by Baudisch in \cite{baudisch-neo} is not pseudofinite.
\end{enumerate}
\end{corollary}

\begin{proof} Case (i) follows immediately from the corollary by the above axiomatisation of  $T_{\nil2,p}$. The groups $D(n)$ in (ii) are built by an amalgamation argument,  working with the class of finite nilpotent class~2 exponent $p$ groups $G$ expanded by linearly independent constants $c_1,\ldots,c_n$ such that $G'\leq \langle c_1,\ldots,c_n\rangle \leq Z(G)$. 
 These also satisfy $\sigma_1\wedge \sigma_2$; this is perhaps  most easily verified by working in the category $\mathbb{B}^P$ with objects $(V,W,\beta)$ where $V,W$ are vector spaces over $\mathbb{F}_p$ and $\beta:V\times V \to W$ is an alternating bilinear form, as described in \cite[Section 3]{baudisch-neo}; we may take $W=\langle c_1,\ldots,c_n\rangle$. 
\end{proof}

\begin{remark}\rm 
The counting argument in Lemma~\ref{bil} is quite extravagant, and does not seem to require the full strength of $\psi$. This suggests a conceivable route to a counterexample to our conjecture, by building a group with sparse failure of $\sigma_2$, possibly exploiting the Shelah-Spencer zero-one laws from \cite{spencer}. 

\end{remark}

We now turn to Apps's more general class of {\em comprehensive} groups, which we briefly summarise. First recall that if $A$ is an abelian $p$-group and $g\in A$, then the {\em height} $\h_A(g)$ is defined to be
$\sup\{n\geq 0:g=h^{p^n} \mbox{~for some~}h\in A\}$. A subgroup $B$ of $A$ is {\em pure in $A$} if $\h_B(g)=\h_A(g)$ for all $g\in B$. For a group $A$ of finite exponent, write $\exp(A)$ for the exponent of $A$. Building on work in \cite{sar-wood}, for $p$ a prime  Apps \cite{apps1} defines a nilpotent class~2 $p$-group $G$ with centre $Z$ to be {\em comprehensive} if it satisfies the following condition, where $|g|$ denotes the order of a group element $g$, $G^*=G/Z$, and for $g\in G$, $g^*$ denotes its image in $G^*$:\\

\noindent
($*$)  \hspace{.5cm} Let $A$ be a finite pure subgroup of $G^*$, let $\alpha\in \Hom(A,Z)$, let $w\in Z$ and $r\in \mathbb{N}$, and suppose $\exp(\alpha(A))\leq p^r\leq \exp(G^*)$ and $p^r|w|\leq \exp(G)$. Then there is $g\in G$ such that 
$\langle A,g^*\rangle$ is pure in $G^*$, $|g^*|=p^r, g^{p^r}=w$, and
$[a, g^*]=\alpha(a)$ for all $a\in A$. \\

Given an abelian $p$-group $Z$ and $t,u\in \mathbb{N}$, Apps defines $\mathcal{N}_2(p^u,p^t/Z)$ to be the class of groups $G$ of exponent $p^u$ with centre $Z$ such that $\exp(G^*)=p^t$. He proves the following (restricted below to finite exponent, the area of our interest).

\begin{theorem}\cite[Theorem C]{apps1} \label{comp}
Let $p$ be prime, $Z$ a countable abelian group of finite exponent $p^s$ and suppose $t,u\in \mathbb{N}$ with $s+t\geq u\geq s\geq t$, with also $u> t$ if $p=2$. Then there is a unique countable comprehensive group $G$  in $\mathcal{N}_2(p^u,p^t/Z)$, and $G$ is $\omega$-categorical. 
\end{theorem}

It can be checked that for $p$ an odd prime, $s=t=u=1$, and $Z$ an elementary abelian $p$-group of rank $n$, the corresponding comprehensive group is exactly Baudisch's group $D(n)$. 

\begin{theorem}\label{comp_not_psf} Let $Z,p, s,t,u$ and the comprehensive group $G$ be as in Theorem~\ref{comp}.  Assume $Z$ contains an elementary abelian $p$-group of rank 2. Then $G$ is not pseudofinite.
\end{theorem}

\begin{proof} Suppose for a contradiction that $G$ is pseudofinite.
Let $G_0^*$ be the subgroup of $G^*$ consisting of elements of height 0  and order  $p$, together with the identity. By \cite[Lemma 3.2(a)]{apps1}, the group $G_0^*$ is infinite.  Let $G_0$ be the preimage of $G_0^*$ under the map $G \to G^*$. Clearly $G_0$ is characteristic in $G$, so is $\emptyset$-definable and so is $\omega$-categorical and also pseudofinite. Let $Z_0$ be the subgroup of $Z$ consisting of elements of order at most $p$. Then $Z_0$ is an elementary abelian $p$-group of rank at least 2. 
Since $[x,y]^p=[x^p,y]$ for any $x,y\in G_0$ by Remark~\ref{commutators}, the commutator map $G_0 \times G_0 \to Z$  induces a bilinear map $\beta:G_0^* \times G_0^* \to Z_0$, and the structure $(G_0^*, Z_0, \beta)$ is definable, and so $\omega$-categorical and pseudofinite. Furthermore, it follows from ($*$) that for any linearly independent $g_1,g_2\in G_0^*$ and $h_1,h_2\in Z_0$, there is $k\in G_0^*$ with $\beta(g_i,k)=h_i$ for $i=1,2$. This however is impossible by Lemma~\ref{bil}.


\end{proof}

\begin{remark} \rm
Let $p$ be an odd prime and $m=p^u$ for some $u\geq 1$. 
It is shown at the end of Section 3 of \cite{apps1} that any countable group satisfying the model companion of the theory of class~2 groups of exponent $m$ is a comprehensive group with respect to $\mathcal{N}_2(p^u,p^u/Z)$,
where $Z=(C_{p^u})^\omega$. Since such $Z$ contains an infinite elementary abelian $p$-group, the model companion is not pseudofinite. 
\end{remark}

As mentioned in the introduction, for any odd prime $p$ and $c<p$, Maier \cite{maier} constructs an $\omega$-categorical model companion of the class of exponent $p$ groups of nilpotency class at most $c$. The construction is described in detail in \cite{DMRS}, via an intricate amalgamation of `Lazard' Lie algebras which yields additional information on the model theory of the groups. The countable model of the model companion is denoted $\mathbf{G}_{c,p}$. It turns out (personal communication of Christian d'Elb\'ee) that if $c\geq 4$ then
 $\mathbf{G}_{c,p}$ interprets $\mathbf{G}_{2,p}$, which is the countable model of $T_{\nil2,p}$; the argument provided by d'Elb\'ee does not handle the case $c=3$.   In particular, by Corollary~\ref{nil-pseud}, we have

\begin{proposition}
Let $p$ be an odd prime and $3<c<p$. Then $\mathbf{G}_{c,p}$ is not pseudofinite.
\end{proposition}

We have not tried systematically to show that all known groups of this flavour (i.e. $\omega$-categorical model
 companions, Fra\"iss\'e limits in various languages) satisfy Conjecture~\ref{mainq}. In particular, we only work with odd primes, so for example do not consider the QE nilpotent class 2 groups of exponent 4 mentioned in the introduction. 

We finish this section with a brief discussion of which FAF groups are pseudofinite.
\begin{examples}\label{examples} \rm
 First, by \cite[Theorem 63]{bcm}, every $\omega$-categorical abelian-by-finite group has $\omega$-stable theory, and hence is smoothly approximable (see e.g. \cite[Corollary 7.4]{chl}) and so pseudofinite. As noted in Section 1, extraspecial $p$-groups of odd exponent $p$ are pseudofinite, whereas by Theorem~\ref{comp_not_psf} for finite $n\geq 2$ Baudisch's groups $D(n)$ are not; the groups in both of these classes are finite-by-abelian. At the other extreme to the groups $D(n)$, consider (for an odd prime $p$ and any finite $n\geq 1$) a vector space $V$ over $\mathbb{F}_{p^n}$ equipped with a non-degenerate symplectic form $\beta: V\times V \to \mathbb{F}_{p^n}$. Define the group $H(p,n)$ to have universe 
$V\times \mathbb{F}_{p^n}$, with group operation $(v_1,x_1)*(v_2,x_2)=(v_1+v_2, x_1+x_2+\beta(v_1,v_2))$. Then $H(p,n)$ is bi-interpretable with a smoothly approximable structure and so is smoothly approximable and hence pseudofinite (unlike $D(n)$). However, like $D(n)$, $H(p,n)$ is nilpotent class~2 of exponent $p$ with centre elementary abelian of rank $n$, and is supersimple of SU-rank 1. If $n=1$ then $H(p,n)$ is extraspecial (so is isomorphic to $D(1)$). 

The groups $H(p,n)$ are a special case of a more general class of $\omega$-categorical nilpotent class~2 groups discussed by Apps in \cite[Section 2]{apps1}. Namely, given a finite nilpotent class~2 group $G$, and subgroup $K$ of $G$ with $G'\leq K\leq Z(G)$, let $G(\omega;K)$ denote the (restricted) central product of $\aleph_0$ copies of $G$, amalgamated over $K$, and for $n\in \mathbb{N}$ with $n>1$ let $G(n;K)$ be the central product of $n$ copies of $G$ amalgamated over $K$. Apps shows in \cite[Theorem A]{apps1} that $G(\omega;K)$ is $\omega$-categorical.

\end{examples}

\begin{question} Is it true that for any  $G,K$ as above, the group $G(\omega,K)$ is pseudofinite and satisfies the limit theory of the groups $G(n,K)$? Are all such groups smoothly approximated by the $G(n,K)$ (and hence supersimple of finite rank)?
\end{question}

Observe that if $G$ is a finite non-abelian group, then the restricted {\em direct} power $H:=G^\omega$ is not pseudofinite. For define a preorder $x<y$ on $H$ by putting $x<y$ if and only if $C_H(x)>C_H(y)$. Then $H\models (\forall x\in H\setminus Z(H)) (\exists y\in H\setminus Z(H))(x<y)$; indeed, we could choose $y$ to agree with $x$ on its support, and have a non-central element of $G$ in another entry. Such a sentence  could not hold in any finite non-abelian group. In fact, such a group $H$ is also not $\omega$-categorical -- see e.g. \cite[Theorem 3]{ros}.

\begin{question}
Does every $\omega$-categorical FAF group have supersimple theory?
\end{question}

\section{Further observations}

We begin by noting that the Boolean power construction mentioned in the introduction does not seem  to give counterexamples to Conjecture~\ref{mainq}. Note that in the lemma below, if $G$ is nilpotent of class 2 then so are $B[G]$ and $B^-[G]$. 

\begin{lemma}\label{boolpseud}
Let $G$ be an $\omega$-categorical  non-abelian group and let $B[G]$ or $B^-[G]$ be an $\omega$-categorical Boolean power of $G$. Then $B[G]$ (respectively $B^-[G]$) has the strict order property so is not pseudofinite. 
\end{lemma}
\begin{proof}
We consider the case when $B$ is the countable atomless Boolean algebra, but the other cases (where $B$ has finitely many atoms, or is an $\omega$-categorical Boolean ring without 1) are similar. Let $C$ be the Cantor set. Then 
$B[G]$ consists of all continuous maps $C\to G$ with finite support. Let $\{U_i:i\in \omega\}$ be clopen subsets of $C$, with $U_i\subset U_j$ whenever $i<j$. Also let $g\in G\setminus Z(G)$ and let $\phi_i:C\to G$ be the map taking value $g$ on $U_i$ and $1$ elsewhere.  Then for each $i$ we have $\phi_i\in B[G]$, and $C_{B[G]}(\phi_i)$  consists of elements $\psi$ of $B[G]$ taking any value on $C\setminus U_i$ and taking values in $C_G(g)$ on $U_i$. Hence if $i>j$ we have $C_G(\phi_i)<C_G(\phi_j)$. Thus, the formula $\phi(x,y)$ expressing $C_{B[G]}(x)\leq C_{B[G]}(y)$ defines a preorder with an infinite totally ordered subset. Non-pseudofiniteness now follows from Lemma~\ref{SOP}.
\end{proof}

We also pose the following question, aiming to recover a Boolean power structure for a certain class of nilpotent class~2 groups, mimicking the proof of  Theorem A of \cite{apps2}  which stems ultimately from an unpublished result of Philip Hall.
\begin{problem}
Find conditions on a locally finite nilpotent class~2 (perhaps $\omega$-categorical)  group $G$ which guarantee that $G$ is a Boolean power. Suggested conditions are
 (i) there is no non-trivial characteristic subgroup of $G$ which is properly contained in the derived subgroup $G'$, and (ii) $G$ has a subgroup of finite index not containing $G'$. Find conditions which ensure that $G$ is a filtered Boolean power in the sense of \cite{mr}.
\end{problem}

Let $f(n)$ be the number of non-isomorphic groups of order $n$. It is well-known that $f(n)$ grows fast for prime power $n$, in the sense that for prime $p$ we have $f(p^m)\geq p^{\frac{2}{27}m^2(m-6)}$ (see \cite{higman0}). Given this ubiquity of finite $p$-groups, we expect there to be many $\omega$-categorical class~2 $p$-groups not of the types described above. Note, though, that the  very flexible `Mekler construction' (see \cite{mekler}), though it gives examples of pseudofinite groups (see \cite{mac-tent}), does not preserve $\omega$-categoricity so is not likely to give new examples. 

The deep theory in \cite{cher-hrush} yields the following consequence. 
\begin{proposition} Suppose that $d\in \mathbb{N}$ and $\mathcal{C}$ is a family of finite groups such that, for all $G\in \mathcal{C}$, the group $\Aut(G)$ has at most $d$ orbits on $G^4$. Then  the following hold.

\begin{enumerate}
\item[(i)] There is $e\in \mathbb{N}$ such that each $G\in \mathcal{C}$ has a normal subgroup $N$ of index at most $e$, with $N'\leq Z(N)$ and $|N'|\leq e$.
\item[(ii)] Any ultraproduct of members of $\mathcal{C}$ is FAF with $\omega$-categorical theory.
\end{enumerate}
\end{proposition}

\begin{proof}
\begin{enumerate}
\item[(i)] It follows from the structure theory developed in \cite{cher-hrush} (see e.g. Theorem 6) that the (sufficiently large) members of $\mathcal{C}$ fall into finitely many families $\mathcal{C}_i$ each associated with an $\omega$-categorical countably infinite Lie coordinatizable group $G_i$ whose theory is the collection of sentences which hold in all but finitely many groups in $\mathcal{C}_i$. In particular, by \cite[Theorem 7]{cher-hrush} such a group  is modular and of finite rank  (the rank used in \cite{cher-hrush} can be taken after the fact to be SU-rank). By \cite[Proposition 6.2.4]{cher-hrush} the group $G_i$ is FAF, and the result follows.

\item[(ii)] This is immediate from (i).
\end{enumerate}
\end{proof}
Thus, for any construction of an $\omega$-categorical pseudofinite not virtually  finite-by-abelian group, the $\omega$-categoricity will not arise directly from richness of the automorphism groups of a class of finite groups. Such a construction  will have to involve some model-theoretic argument to obtain $\omega$-categoricity (e.g. back-and-forth, Fra\"iss\'e amalgamation, or axiomatising an appropriate class). We cannot envisage such a construction which eludes the argument against pseudofiniteness given in Section 2. 

Finally, we remark that the structure theory for $\omega$-categorical rings (not necessarily commutative) has close parallels to that of $\omega$-categorical groups. Thus, we have the following conjecture and proposition. Recall that a ring $R$ is {\em  nilpotent} (of class $r$) if $x_1\ldots x_r=0$ for any $x_1,\ldots,x_r\in R$. We say $R$ is {\em null} if $xy=0$ for any $x,y\in R$. 

\begin{conjecture} \label{qrings} If $R$ be an $\omega$-categorical pseudofinite ring, then $R$ has a definable 2-sided ideal $I$ of finite index which is finite-by-null, i.e. $I$ has a finite  2-sided ideal $J$ such that $I/J$ is null.
\end{conjecture}

We remark that  every supersimple $\omega$-categorical ring is in this sense finite-by-null-by-finite (see \cite[Theorem 3.4]{krupwag}, and its proof for the definability assertion).   

\begin{proposition}
Let $R$ be an $\omega$-categorical pseudofinite ring. Then $R$ has a definable nilpotent ideal of finite index.
\end{proposition}

\begin{proof} By \cite[Theorem 3.1]{krup}, any $\omega$-categorical  ring whose theory does not have the strict order property  has a definable nilpotent ideal of finite index (the definability clause comes from the proof in \cite{krup}, using the corresponding definability in Corollary~\ref{SOPgroup}). The result now follows from Lemma~\ref{SOP}.

\end{proof}


\end{document}